\documentclass[a4paper,12pt,reqno]{amsart}
\usepackage{graphicx}
\usepackage{amsmath,amssymb,cite,mathrsfs}

\numberwithin{equation}{section}
\newtheorem{theorem}{Theorem}[section]
\newtheorem{lemma}[theorem]{Lemma}
\newtheorem{prop}[theorem]{Proposition}

\theoremstyle{remark}
\newtheorem{remark}[theorem]{Remark}

\long\def\comment#1{}
\numberwithin{equation}{section}
\allowdisplaybreaks

\begin{document}

\title{Mean value theorems for the double zeta-function}


\author{Kohji Matsumoto}
\address{K. Matsumoto:\ Graduate School of Mathematics, Nagoya University, Chikusa-ku, Nagoya 464-8602 Japan}
\email{kohjimat@math.nagoya-u.ac.jp}

\author{Hirofumi Tsumura}
\address{H. Tsumura:\ Department of Mathematics and Information Sciences, Tokyo Metropolitan University, 1-1, Minami-Ohsawa, Hachioji, Tokyo 192-0397 Japan}
\email{tsumura@tmu.ac.jp}



\date{}

\begin{abstract}
We prove asymptotic formulas for mean square values of the Euler double
zeta-function $\zeta_2(s_0,s)$, with respect to $\Im s$.   Those formulas
enable us to propose a double analogue of the Lindel{\"o}f hypothesis.
\end{abstract}

\maketitle

\ 

\section{Introduction and the statement of results} \label{Sec-1}

Let $\mathbb{N}$ be the set of natural numbers, $\mathbb{N}_0:=\mathbb{N}\cup \{0\}$, $\mathbb{Z}$ the ring of rational integers, $\mathbb{Q}$ the field of rational numbers, $\mathbb{R}$ the field of real numbers, $\mathbb{C}$ the field of complex numbers and $i=\sqrt{-1}$. 

The Euler double zeta-function is defined by 
\begin{equation}
\label{1-1}
\zeta_2(s_1,s_2)=\sum_{m=1}^\infty \frac{1}{m^{s_1}}\sum_{n=1}^\infty \frac{1}{(m+n)^{s_2}}=\sum_{k=2}^\infty \left(\sum_{m=1}^{k-1} \frac{1}{m^{s_1}}\right)\frac{1}{k^{s_2}}
\end{equation}
which is absolutely convergent
for $s_1,s_2 \in \mathbb{C}$ with $\Re s_2> 1$ and $\Re (s_1+s_2)>2$ 
(Theorem 3 in \cite{M1}), and can be continued meromorphically to $\mathbb{C}^2$. 
The singularities are $s_2=1$ and $s_1+s_2=2,1,0,-2,-4,\ldots$
(Theorem 1 in \cite{AET}).
Euler himself considered the behaviour of this function when $s_1,s_2$ are positive
integers.   It was Atkinson \cite{Atk} who first studied \eqref{1-1} from the
analytic viewpoint, and he proved the analytic continuation of it.
Recently the active research of \eqref{1-1} revived, because it is the simplest
example of multiple zeta-functions.   As for the studies on the analytic side of
\eqref{1-1}, for example, upper-bound estimates were discussed in \cite{IM,KT,KTZ}, and functional equations were discovered in \cite{KMT-Debrecen,Mat-Camb}.

It is the purpose of the present paper to prove certain mean square formulas for
\eqref{1-1}.
Let
\begin{equation}
\label{1-2}
\zeta_2^{[2]}(s_1,s_2)=\sum_{k=2}^\infty \left|\sum_{m=1}^{k-1} \frac{1}{m^{s_1}}\right|^2\frac{1}{k^{s_2}}.
\end{equation}
Since the inner sum is $O(1)$ (if $\Re s_1>1$), $O(\log k)$ (if $\Re s_1=1$), or
$O(k^{1-\Re s_1})$ (if $\Re s_1<1$), the series \eqref{1-2} is convergent when 
$\Re s_1\geq 1$ and $\Re s_2>1$, or when $\Re s_1<1$ and $2\Re s_1+\Re s_2>3$. 
Note that $\zeta_2^{[2]}(1,q)$ $(q\in \mathbb{N}_{\geq 2})$ was already studied by Borwein et al. (see \cite{BBG}). 

Hereafter we write $s_0$ and $s$ instead of $s_1$ and $s_2$, respectively, and
consider the mean square with respect to $s$, while $s_0$ is to be fixed.

\begin{theorem} \label{T-1-1} For $s_0={\sigma_0}+i{t_0}\in \mathbb{C}$ with ${\sigma_0}>1$ and $s=\sigma+it\in \mathbb{C}$ with $\sigma>1$, $t\geq 2$, we have 
\begin{equation}
\begin{split}
& \int_{2}^{T}|\zeta_2(s_0,s)|^2 dt =\zeta_2^{[2]}(s_0,2\sigma)T+O(1)\qquad (T\to \infty).
\end{split}
\label{1-3}
\end{equation}
\end{theorem}

\begin{theorem} \label{T-1-2} For $s_0={\sigma_0}+i{t_0}\in \mathbb{C}$ with ${\sigma_0}>1$ and $s=\sigma+it\in \mathbb{C}$ with $\frac{1}{2}<\sigma\leq 1$, $t\geq 2$ and
$\sigma_0+\sigma>2$, we have
\begin{equation}
\begin{split}
& \int_{2}^{T}|\zeta_2(s_0,s)|^2 dt =\zeta_2^{[2]}(s_0,2\sigma)T+O\left( T^{2-2\sigma}\log T\right)+O\left( T^{1/2}\right).
\end{split}
\label{1-4}
\end{equation}
\end{theorem}

The most important result in the present paper is the following Theorem \ref{T-1-3},
which describes the situation under the condition $\frac{3}{2}<\sigma_0+\sigma\leq 2$.

\begin{theorem} \label{T-1-3} Let 
$s_0={\sigma_0}+i{t_0} \in \mathbb{C}$ with 
$\frac{1}{2}<{\sigma_0}< \frac{3}{2}$ and $s=\sigma+it\in \mathbb{C}$ with 
$\frac{1}{2}<\sigma\leq 1$, $t\geq 2$ and $\frac{3}{2}<{\sigma_0}+\sigma\leq 2$.
Assume that when $t$ moves from $2$ to $T$, the point $(s_0,s)$ does not encounter
the hyperplane $s_0+s=2$ {\rm (}which is a singular locus of $\zeta_2${\rm )}.   Then 
\begin{equation}
\begin{split}
& \int_{2}^{T}|\zeta_2(s_0,s)|^2 dt =\zeta_2^{[2]}(s_0,2\sigma)T\\
& \quad +
\begin{cases}  
O\left(T^{4-2\sigma_0-2\sigma}\log T\right)+O\left(T^{1/2}\right) 
    & (\frac{1}{2}<\sigma_0<1,\frac{1}{2}<\sigma<1)\\  
O\left(T^{2-2\sigma_0}(\log T)^2\right)+O\left(T^{1/2}\right)   
    & (\frac{1}{2}<\sigma_0<1,\sigma=1)\\                      
O\left(T^{2-2\sigma}(\log T)^3\right)+O\left(T^{1/2}\right)  
    & (\sigma_0=1,\frac{1}{2}<\sigma<1)\\                           
O\left(T^{1/2}\right)               & (\sigma_0=1,\sigma=1)\\                         
O\left(T^{2-2\sigma}\log T\right)+O\left(T^{1/2}\right)           
    & (1<\sigma_0<\frac{3}{2},\frac{1}{2}<\sigma<1). 
\end{cases} 
\end{split}
\label{1-5}
\end{equation}
\end{theorem}

\begin{remark}
In Theorems \ref{T-1-2} and \ref{T-1-3}, the error terms $O(T^{1/2})$ are coming
from the simple application of the Cauchy-Schwarz inequality.
It is plausible to expect that we can reduce these error terms by more 
elaborate analysis.
\end{remark}

It is interesting to compare our theorems with the classical results on the mean
square of the Riemann zeta-function $\zeta(s)$.   It is known that
\begin{align}\label{1-6}
\int_2^T|\zeta(\sigma+it)|^2 dt\sim\zeta(2\sigma)T\qquad\left(\sigma>\frac{1}{2}
     \right)
\end{align}
and
\begin{align}\label{1-7}
\int_2^T|\zeta(\frac{1}{2}+it)|^2 dt\sim T\log T
\end{align}
(see Titchmarsh \cite[Theorems 7.2, 7.3]{Titch}).   These simple results suggest
two important observations.   

(a) First, it is trivial that $\zeta(\sigma+it)$ is
bounded with respect to $t$ in the region of absolute convergence $\sigma>1$,
but \eqref{1-6} and \eqref{1-7} suggest that $\zeta(\sigma+it)$ seems not so large
in the strip $1/2\leq\sigma\leq 1$, too.   In fact, the well-known Lindel{\" o}f
hypothesis predicts that
\begin{align}\label{1-8}
\zeta(\sigma+it)=O\left(t^{\varepsilon}\right)\qquad\left( \frac{1}{2}\leq\sigma<1\right)
\end{align}
for any $\varepsilon>0$.   (For $\sigma=1$, even a stronger estimate has already
been known.)   Formulas \eqref{1-6} and \eqref{1-7} support this hypothesis.

(b) The second observation is that the coefficient $\zeta(2\sigma)$ on the right-hand
side of \eqref{1-6} tends to infinity as $\sigma\to 1/2$, hence the form of the
formula should be changed at $\sigma=1/2$, which is in fact embodied by \eqref{1-7}.
This is one of the special features of the ``critical line'' $\Re s=1/2$ in the theory 
of the Riemann zeta-function.

Our theorems proved in the present paper may be regarded as double analogues of
\eqref{1-6}.   Since the coefficient $\zeta_2^{[2]}(s_0,2\sigma)$ tends to
infinity as $\sigma_0+\sigma\to 3/2$, it is natural to raise,
analogously to the above (a) and (b), the following two conjectures:

(i) (a double analogue of the Lindel{\" o}f hypothesis) For any $\varepsilon>0$, 
\begin{align}\label{1-9}
\zeta_2(s_0,s)=O\left(t^{\varepsilon}\right)
\end{align}
when $(s_0,s)$ (which is not in the domain of absolute convergence) satisfies
$\sigma_0>1/2$, $\sigma>1/2$, $t\geq 2$, $\sigma_0+\sigma\geq 3/2$ and
$s_0+s\neq 2$;

(ii) (the criticality of $\sigma_0+\sigma=3/2$)
When $\sigma_0+\sigma=3/2$, the form of the main term of the mean square formula 
would not be $CT$ (with a constant $C$; most probably, some log-factor would appear).

\begin{remark}
It is not easy to find the ``correct'' double analogue of the Lindel{\" o}f 
hypothesis.   Nakamura and Pa{\'n}kowski \cite{Nak-Pan} raised the conjecture
\begin{align}\label{1-10}
\zeta_2\left(1/2+it,1/2+it\right)=O\left(t^{\varepsilon}\right)
\end{align}
(actually they stated their conjecture for more general multiple case), and gave
a certain result (their Proposition 6.3) which supports the conjecture.
However, the value $\zeta_2\left(1/2+it_1,1/2+it_2\right)$ is, if $t_1\neq t_2$,
not always small.   In fact, Corollary 1 of Kiuchi, Tanigawa and Zhai \cite{KTZ}
describes the situation when $\zeta_2(s_1,s_2)$ is not small.   For example,
if $t_2\ll t_1^{1/6-\varepsilon}$, then
$$
\zeta_2\left(1/2+it_1,1/2+it_2\right)=\Omega\left(t_1^{1/3+\varepsilon}\right).
$$
Our theorems imply that our conjecture \eqref{1-9} is true in mean.   That is,
\eqref{1-9} is reasonable in view of our theorems.
\end{remark}

\begin{remark}
The above conjecture (ii) suggests that $\sigma_0+\sigma=3/2$ might be the
double analogue of the critical line of the Riemann zeta-function $\Re s=1/2$.
On the other hand, in view of the result of Nakamura and Pa{\'n}kowski mentioned
above, we see that another candidate of the double analogue of the critical line
is $\sigma_0+\sigma=1$.   At present it is not clear which is more plausible.
\end{remark}

\begin{remark}
We cannot expect the analogue of the Riemann hypothesis on the location
of zeros.   In fact, Theorem 5.1 of Nakamura and Pa{\'n}kowski \cite{Nak-Pan} 
asserts (in the double zeta case) that for any $1/2<\sigma_1<\sigma_2<1$,
$\zeta_2(s,s)$ has $\asymp T$ non-trivial zeros in the rectangle
$\sigma_1<\sigma<\sigma_2$, $0<t<T$.
\end{remark}

The plan of the present paper is as follows.   We first prove the simplest
Theorem \ref{T-1-1} in Section \ref{sec-2}.   To prove the other theorems, we need
certain approximation formulas for $\zeta_2(s_0,s)$.   Using the Euler-Maclaurin
formula, we show the first approximation formula (Theorem \ref{T-3-1}) 
in Section \ref{sec-3}, and using
it, we prove Theorem \ref{T-1-2} in Section \ref{sec-4}.   In Section \ref{sec-4.25}
we introduce and discuss the double analogue of the Euler constant.   The most
difficult part of the present paper is the proof of Theorem \ref{T-1-3}.
In Section \ref{sec-4.5} we show the second approximation formula (Theorem
\ref{T-5-3}), by employing
the method of Mellin-Barnes integral formula.   Based on this second approximation 
formula, we give the proof of Theorem \ref{T-1-3} in the final Section \ref{sec-5}.

A possible direction of future study is to search for a strong type of
approximate functional equation (that is, similar to 
\cite[Theorem 4.16]{Titch})
for the double zeta-function, based on our previous results
on functional equations for the double zeta-function obtained in
\cite{KMT-Debrecen,Mat-Camb}. If we could
succeed in finding such an equation, we would be able to give
a more precise version of
mean value theorems for the double zeta-function.

A part of the results in this paper has been announced in \cite{MT-RIMS}.

\section{Proof of Theorem \ref{T-1-1}} \label{sec-2}

In this section, we give the proof of Theorem \ref{T-1-1}. Throughout this paper, we frequently use the following elementary estimations:
\begin{align*}
& \sum_{m=1}^{k-1} \frac{1}{m}  \ll \int_{1}^{k} u^{-1}du = \log k,\\
& \sum_{m=1}^{k-1} \frac{1}{m^\sigma}  \ll \int_{0}^{k} u^{-\sigma}du = \frac{k^{1-\sigma}}{1-\sigma}\qquad (0<\sigma<1),\\
& \sum_{m=k}^{\infty} \frac{1}{m^{\sigma}}  \ll \int_{k}^{\infty} u^{-\sigma}du = \frac{k^{1-\sigma}}{\sigma-1}\qquad (\sigma>1).
\end{align*}

\begin{proof}[Proof of Theorem \ref{T-1-1}]
Let $s_0={\sigma_0}+i{t_0}\in \mathbb{C}$ with ${\sigma_0}>1$ and $s=\sigma+it\in \mathbb{C}$ with $\sigma>1$. We set
\begin{align*}
S:=\zeta_2(s_0,s) \overline{\zeta_2(s_0,{s})} & =\sum_{m_1\geq 1 \atop n_1\geq 1}\frac{1}{m_1^{s_0}(m_1+n_1)^{\sigma+it}}\sum_{m_2\geq 1 \atop n_2\geq 1}\frac{1}{m_2^{\overline{s_0}}(m_2+n_2)^{\sigma-it}}.
\end{align*}
Taking out the terms corresponding to $m_1+n_1=m_2+n_2$ and setting $k=m_1+n_1$, we have
\begin{align*}
S& =\sum_{k=2}^\infty \left( \sum_{m_1=1}^{k-1} \sum_{m_2=1}^{k-1}\frac{1}{m_1^{s_0} m_2^{\overline{s_0}}}\right) \frac{1}{k^{2\sigma}}\\
& \ +\sum_{m_1,m_2,n_1,n_2\geq 1 \atop m_1+n_1\not= m_2+n_2}\frac{1}{m_1^{s_0}m_2^{\overline{s_0}}(m_1+n_1)^{\sigma+it}(m_2+n_2)^{\sigma-it}}\\
& =\zeta_2^{[2]}({s_0},2\sigma) +\sum_{m_1,m_2,n_1,n_2\geq 1 \atop m_1+n_1\not= m_2+n_2}\frac{1}{m_1^{s_0}m_2^{\overline{s_0}}(m_1+n_1)^{\sigma}(m_2+n_2)^{\sigma}}\left(\frac{m_2+n_2}{m_1+n_1}\right)^{it}.
\end{align*}
Hence we have
\begin{align*}
&\int_{2}^{T}|\zeta_2({s_0},s)|^2 dt =\zeta_2^{[2]}({s_0},2\sigma)(T-2) \\
&\qquad + \sum_{m_1,m_2,n_1,n_2\geq 1 \atop m_1+n_1\not= m_2+n_2}\frac{1}{m_1^{s_0}m_2^{\overline{s_0}}(m_1+n_1)^{\sigma}(m_2+n_2)^{\sigma}}\int_{2}^{T}\left(\frac{m_2+n_2}{m_1+n_1}\right)^{it}dt.
\end{align*}
The second term on the right-hand side is 
\begin{align*}
& \sum_{m_1,m_2,n_1,n_2\geq 1 \atop m_1+n_1\not= m_2+n_2}\frac{1}{m_1^{s_0}m_2^{\overline{s_0}}(m_1+n_1)^{\sigma}(m_2+n_2)^{\sigma}}\\
& \qquad \times \frac{e^{iT\log((m_2+n_2)/(m_1+n_1))}-e^{2i\log((m_2+n_2)/(m_1+n_1))}}{i\log((m_2+n_2)/(m_1+n_1))}\\
& \ll \sum_{m_1,m_2,n_1,n_2\geq 1 \atop m_1+n_1< m_2+n_2}\frac{1}{(m_1m_2)^{{\sigma_0}}(m_1+n_1)^{\sigma}(m_2+n_2)^{\sigma}}\frac{1}{\log\frac{m_2+n_2}{m_1+n_1}}\\
& =\left(\sum_{m_1,m_2,n_1,n_2\geq 1 \atop m_1+n_1< m_2+n_2\leq 2(m_1+n_1)}+\sum_{m_1,m_2,n_1,n_2\geq 1 \atop m_2+n_2> 2(m_1+n_1)}\right)\frac{1}{(m_1m_2)^{{\sigma_0}}}\\
& \qquad \qquad \times \frac{1}{(m_1+n_1)^{\sigma}(m_2+n_2)^{\sigma}\log\frac{m_2+n_2}{m_1+n_1}}.
\end{align*}
We denote the right-hand side by $V_1+V_2$. Then we have
\begin{align*}
V_2 & \ll \sum_{m_1,m_2,n_1,n_2\geq 1 \atop m_2+n_2> 2(m_1+n_1)}\frac{1}{(m_1m_2)^{{\sigma_0}}(m_1+n_1)^{\sigma}(m_2+n_2)^{\sigma}}\\
& \ll \sum_{m_1,m_2,n_1,n_2\geq 1}\frac{1}{(m_1m_2)^{{\sigma_0}}(n_1n_2)^{\sigma}}=O(1).
\end{align*}
As for $V_1$, setting $r=(m_2+n_2)-(m_1+n_1)$, we have
\begin{align*}
V_1&=\sum_{m_1,m_2,n_1\geq 1}\frac{1}{(m_1m_2)^{{\sigma_0}}}\sum_{r=1}^{m_1+n_1}\frac{1}{(m_1+n_1)^{\sigma}(m_1+n_1+r)^{\sigma}}\frac{1}{\log\frac{m_1+n_1+r}{m_1+n_1}}.
\end{align*}
Since $m_1+n_1+r \asymp m_1+n_1$, we obtain
\begin{align*}
V_1&\ll \sum_{m_1,m_2,n_1\geq 1}\frac{1}{(m_1m_2)^{{\sigma_0}}}\frac{1}{(m_1+n_1)^{2\sigma}}\sum_{r=1}^{m_1+n_1}\frac{1}{\log\left(1+\frac{r}{m_1+n_1}\right)}\\
&\ll \sum_{m_1,m_2,n_1\geq 1}\frac{1}{(m_1m_2)^{{\sigma_0}}}\frac{1}{(m_1+n_1)^{2\sigma}}\sum_{r=1}^{m_1+n_1}\frac{m_1+n_1}{r}\\
&\ll \sum_{m_1,m_2,n_1\geq 1}\frac{1}{(m_1m_2)^{{\sigma_0}}}\frac{1}{(m_1+n_1)^{2\sigma-1}}\log(m_1+n_1)\\
&\ll \sum_{m_2\geq 1}\frac{1}{m_2^{\sigma_0}}\sum_{m_1,n_1\geq 1}\frac{\log(m_1+n_1)}{m_1^{{\sigma_0}}(m_1+n_1)^{2\sigma-1}}=O(1),
\end{align*}
because ${\sigma_0}>1$ and $\sigma>1$.
This completes the proof of Theorem \ref{T-1-1}.
\end{proof}

\begin{remark}
The fundamental idea of the above proof of Theorem \ref{T-1-1} is similar to that
of the proof of \cite[Theorem 7.2]{Titch}.   The basic structure of the proofs of
Theorems \ref{T-1-2} and \ref{T-1-3} given below is the same, though the technical
details are more complicated.
\end{remark}

\section{The first approximation theorem} \label{sec-3} 

Hardy and Littlewood proved the following well-known result (see \cite[Theorem 4.11]{Titch}). Let $\sigma_1>0$, $x\geq 1$ and $C>1$. Suppose $s=\sigma+it \in \mathbb{C}$ with $\sigma\geq \sigma_1$ and $|t|\leq 2\pi x/C$.   Then 
\begin{equation}
\label{3-1}
\zeta(s)=\sum_{1\leq n \leq x}\frac{1}{n^{s}}-\frac{x^{1-s}}{1-s}+O\left( x^{-\sigma}\right)\quad (x\to \infty).
\end{equation}

Here we prove the double series analogue of \eqref{3-1} as follows.

\begin{theorem} \label{T-3-1} 
Let ${s_0}={\sigma_0}+i{t_0} \in \mathbb{C}$,
$s=\sigma+it \in \mathbb{C}\setminus\{1\}$, $x\geq 1$ and $C>1$.  Suppose 
$\sigma> \max(0,2-{\sigma_0})$ and $|t|\leq 2\pi x/C$.   Then 
\begin{equation}
\label{3-2}
\begin{split}
\zeta_2({s_0},s) & =\sum_{m=1}^\infty \sum_{1\leq n \leq x}\frac{1}{m^{{s_0}}(m+n)^{s}}-\frac{1}{1-s}\sum_{m=1}^\infty \frac{1}{m^{s_0} (m+x)^{s-1}}\\
 & \qquad+
\begin{cases}  
   O(x^{-\sigma})   &  (\sigma_0>1)\\
   O(x^{-\sigma}\log x)   &  (\sigma_0=1)\qquad (x \to \infty).\\
   O(x^{1-\sigma-\sigma_0})   &  (\sigma_0<1)
\end{cases} 
\end{split}
\end{equation}
\end{theorem}

In order to prove this theorem, we quote the following lemma.

\begin{lemma}[\cite{Titch} Lemma 4.10] \label{L-3-2}
Let $f(x)$ be a real function with a continuous and steadily decreasing derivative $f'(x)$ in $(a,b)$, and let $f'(b)=\alpha$, $f'(a)=\beta$. Let $g(x)$ be a real positive decreasing function with a continuous derivative $g'(x)$, satisfying that $|g'(x)|$ is steadily decreasing. Then
\begin{equation}
\begin{split}
\sum_{a<n\leq b}g(n)e^{2\pi i f(n)}& =\sum_{\nu \in \mathbb{Z} \atop 
\alpha-\eta<\nu<\beta+\eta}\int_{a}^{b}g(x)e^{2\pi i(f(x)-\nu x)}dx \\
& +O\left( g(a)\log(\beta-\alpha+2)\right)+O\left( |g'(a)|\right)
\end{split}
\label{3-3}
\end{equation}
for an arbitrary $\eta\in (0,1)$. 
\end{lemma}

\begin{proof}[Proof of Theorem \ref{T-3-1}]
By the Euler-Maclaurin formula (see \cite[Equation (2.1.2)]{Titch}), we have
\begin{equation}
\label{EulerMaclaurin}
\begin{split}
\sum_{a<l\leq b}\frac{1}{l^s} & =\frac{b^{1-s}-a^{1-s}}{1-s}-s\int_{a}^{b}\frac{y-[y]-1/2}{y^{s+1}}dy +\frac{1}{2}\left(b^{-s}-a^{-s}\right)
\end{split}
\end{equation}
for $0<a<b$. 
At first assume $\sigma_0>1$, $\sigma>1$.
Setting $a=m+N$ (where $m \in \mathbb{N}$, $N \in \mathbb{N}_0$) in 
\eqref{EulerMaclaurin} and $b\to \infty$, we have
\begin{equation*}
\begin{split}
\sum_{l=m+N+1}^{\infty}\frac{1}{l^s} & =-\frac{(m+N)^{1-s}}{1-s}-s\int_{m+N}^{\infty}\frac{y-[y]-1/2}{y^{s+1}}dy -\frac{1}{2}(m+N)^{-s}.
\end{split}
\end{equation*}
Therefore we have
\begin{align}
& \sum_{m=1}^\infty \frac{1}{m^{s_0}}\sum_{n=1}^\infty \frac{1}{(m+n)^s}\notag\\
&\ =\sum_{m=1}^\infty \frac{1}{m^{s_0}}\sum_{n=1}^{N} \frac{1}{(m+n)^s}
  -\sum_{m=1}^\infty \frac{(m+N)^{1-s}}{m^{s_0}(1-s)}\notag\\
& \quad -s\sum_{m=1}^\infty \frac{1}{m^{s_0}}\int_{m+N}^{\infty}
     \frac{y-[y]-1/2}{y^{s+1}}dy 
 -\frac{1}{2}\sum_{m=1}^\infty \frac{1}{m^{s_0}(m+N)^{s}}\notag\\
&\ =A_1-A_2-A_3-A_4,\label{3-4}
\end{align}
say. The terms $A_1$ and $A_4$ are absolutely convergent in the region 
$\sigma_0+\sigma>1$, and in this region
\begin{align}\label{S_4}
A_4=O\left( \sum_{m=1}^\infty \frac{1}{m^{\sigma_0}(m+N)^{\sigma}}\right).
\end{align}
The integral in $A_3$ is absolutely convergent if $\sigma>0$, and is 
$O(\sigma^{-1}(m+N)^{-\sigma})$.   Therefore $A_3$ can be continued to
the region $\sigma>0$, $\sigma_0+\sigma>1$ and 
\begin{align}\label{S_3}
A_3=O\left( \sum_{m=1}^\infty \frac{|s|/\sigma}{m^{\sigma_0}(m+N)^{\sigma}}\right)
\end{align}
there.   The term $A_2$ is absolutely convergent for $\sigma_0+\sigma>2$, $s\neq 1$.
Therefore we see that the right-hand side of \eqref{3-4} 
gives the meromorphic continuation to the desired region.

Hereafter in this proof we assume $N>x$. 
The term $A_1$ can be rewritten as 
\begin{equation}
\begin{split}
& \sum_{m=1}^\infty \sum_{n\leq x} \frac{1}{m^{s_0}(m+n)^s}+\sum_{m=1}^\infty \sum_{x<n\leq N} \frac{e^{-it\log(m+n)}}{m^{s_0}(m+n)^\sigma}.
\end{split}
\label{3-5}
\end{equation}
Fix $m\in \mathbb{N}$ and set 
$$f(x)=\frac{t}{2\pi}\log(m+x),\quad g(x)=(m+x)^{-\sigma},$$
$(a,b)=(x,N)$ in Lemma \ref{L-3-2}. Then we have 
$$(\alpha,\beta)=\left( \frac{t}{2\pi (m+N)},\ \frac{t}{2\pi (m+x)}\right).$$
We see that
$$|f'(x)|=\frac{|t|}{2\pi (m+x)}\leq \frac{|t|}{2\pi x}\leq \frac{1}{C}<1.$$
When $\sigma>0$, the function $g(x)$ is decreasing and so we can apply Lemma
\ref{L-3-2}.   By taking a small $\eta$, we obtain from \eqref{3-3} that 
\begin{align*}
&\sum_{x<n\leq N} \frac{e^{it\log(m+n)}}{(m+n)^\sigma}=\int_{x}^{N} \frac{1}{(m+u)^{\sigma-it}}du+O\left((m+x)^{-\sigma}\right).
\end{align*}
Considering complex conjugates on the both sides, we have
\begin{align}
\sum_{x<n\leq N}\frac{1}{(m+n)^s}&=
\sum_{x<n\leq N} \frac{e^{-it\log(m+n)}}{(m+n)^\sigma}=\int_{x}^{N} \frac{1}{(m+u)^s}du+O\left((m+x)^{-\sigma}\right)\notag\\
& =\frac{(m+N)^{1-s}-(m+x)^{1-s}}{1-s}+O\left((m+x)^{-\sigma}\right). \label{conjugate}
\end{align}
In other words, denoting the above error term by $E(s;x,m,N)$, we find that
this function is entire in $s$ (the point $s=1$ is a removable singularity)
and satisfies
\begin{align}\label{E}
E(s;x,m,N)=O\left((m+x)^{-\sigma}\right)
\end{align}
uniformly in $N$ in the region $\sigma>0$. 
Using \eqref{conjugate}, we find that the second term of \eqref{3-5} is equal to
\begin{equation}
\begin{split}
& \frac{1}{1-s}\sum_{m=1}^\infty\frac{1}{m^{{s_0}}(m+N)^{s-1}}-
\frac{1}{1-s}\sum_{m=1}^\infty\frac{1}{m^{s_0} (m+x)^{s-1}}\\
& \qquad +\sum_{m=1}^\infty\frac{E(s;x,m,N)}{m^{s_0}}
\end{split}
\label{3-6}
\end{equation}
(where the first two sums are convergent in $\sigma_0+\sigma>2$,
while the last sum is convergent in $\sigma_0+\sigma>1$ because of \eqref{E}),
whose first term is cancelled with $A_2$.
Therefore now we have
\begin{align}
\zeta_2(s_0,s)&=\sum_{m=1}^{\infty}\sum_{n\leq x}\frac{1}{m^{s_0} (m+n)^{s}}
  -\frac{1}{1-s}\sum_{m=1}^\infty\frac{1}{m^{s_0} (m+x)^{s-1}}\notag\\
& \qquad +\sum_{m=1}^\infty\frac{E(s;x,m,N)}{m^{s_0}}-A_3-A_4 \label{suzumenooyado}
\end{align} 
in the region $\sigma>\max(0,2-\sigma_0)$, $s\neq 1$.
Letting $N \to \infty$, and noting \eqref{S_4}, \eqref{S_3} and \eqref{E},
we obtain the proof of Theorem \ref{T-3-1}.
\end{proof}

\section{Proof of Theorem \ref{T-1-2}} \label{sec-4}

In this section, using Theorem \ref{T-3-1}, we give the proof of Theorem \ref{T-1-2}.

\begin{proof}[Proof of Theorem \ref{T-1-2}]
Let ${s_0}={\sigma_0}+i{t_0}\in \mathbb{C}$ with ${\sigma_0}>1$ and 
$s=\sigma+it\in \mathbb{C}\setminus\{1\}$ with $1/2<\sigma\leq 1$, 
$\sigma_0+\sigma>2$.
Setting $C=2\pi$ and $x=t$ 
in \eqref{3-2}, we easily see that the second term on the right-hand side is
$O\left( t^{-\sigma}\right)$, so
we have
\begin{equation}
\label{4-1}
\begin{split}
\zeta_2({s_0},s) & =\sum_{m=1}^\infty \sum_{1\leq n \leq t}\frac{1}{m^{{s_0}}(m+n)^{s}}
+O\left( t^{-\sigma}\right)\quad (t\to \infty).
\end{split}
\end{equation}
We denote the first term on the right-hand side by $\Sigma_1({s_0},s)$. Let $M(n_1,n_2)=\max\{n_1,n_2,2\}$. Then 
\begin{align}
& \int_{2}^{T}|\Sigma_1({s_0},s)|^2 dt \notag\\
& =\int_{2}^{T}\sum_{m_1\geq 1}\sum_{n_1\leq t}\frac{1}{m_1^{{s_0}}(m_1+n_1)^{\sigma+it}}\sum_{m_2\geq 1}\sum_{n_2\leq t}\frac{1}{m_2^{{\overline{s_0}}}(m_2+n_2)^{\sigma-it}} dt\notag\\
& =\sum_{m_1\geq 1}\sum_{m_2\geq 1}\frac{1}{m_1^{s_0}m_2^{\overline{s_0}}}\sum_{n_1\leq T}\sum_{n_2\leq T}\frac{1}{(m_1+n_1)^{\sigma}(m_2+n_2)^{\sigma}}\notag\\
& \qquad \times \int_{M(n_1,n_2)}^{T}\left(\frac{m_2+n_2}{m_1+n_1}\right)^{it}dt\notag\\
& =\sum_{m_1\geq 1}\sum_{m_2\geq 1}\frac{1}{m_1^{s_0}m_2^{\overline{s_0}}}\sum_{n_1\leq T}\sum_{n_2\leq T \atop m_1+n_1=m_2+n_2}\frac{1}{(m_1+n_1)^{2\sigma}}(T-M(n_1,n_2))\notag\\
& \quad + \sum_{m_1\geq 1}\sum_{m_2\geq 1}\frac{1}{m_1^{s_0}m_2^{\overline{s_0}}}\sum_{n_1\leq T}\sum_{n_2\leq T \atop m_1+n_1\not=m_2+n_2}\frac{1}{(m_1+n_1)^{\sigma}(m_2+n_2)^\sigma}\notag\\
& \quad \times \frac{e^{iT\log((m_2+n_2)/(m_1+n_1))}-e^{iM(n_1,n_2)\log((m_2+n_2)/(m_1+n_1))}}{i\log((m_2+n_2)/(m_1+n_1))}.\label{mean-val}
\end{align}
We denote the first and the second term on the right-hand side by $S_1T-S_2$ and $S_3$, respectively. As for $S_1$, setting $k=m_1+n_1(=m_2+n_2)$, we have
\begin{align*}
S_1&= \sum_{k=2}^\infty \left(\sum_{m_1=1}^{k-1} \sum_{m_2=1}^{k-1}\frac{1}{m_1^{s_0}m_2^{\overline{s_0}}}\right)\frac{1}{k^{2\sigma}} \\
& \ -\sum_{m_1\geq 1\atop {m_2\geq 1}}\frac{1}{m_1^{s_0}m_2^{\overline{s_0}}}\bigg\{ \sum_{n_1> T \atop {n_2\leq T \atop m_1+n_1=m_2+n_2}}+\sum_{n_1\leq T \atop {n_2> T \atop m_1+n_1=m_2+n_2}}+\sum_{n_1> T\atop {n_2> T \atop m_1+n_1=m_2+n_2}}\bigg\}\frac{1}{(m_1+n_1)^{2\sigma}}.
\end{align*}
We further denote the second term on the right-hand side by $-(U_1+U_2+U_3)$, which is equal to $-(U_1+U_3)-(\overline{U_1}+U_3)+U_3$ because $U_2=\overline{U_1}$. Since ${\sigma_0}>1$, we have
\begin{align*}
U_1+U_3& \ll \sum_{m_1\geq 1\atop {m_2\geq 1}}\frac{1}{(m_1m_2)^{\sigma_0}} \sum_{n_1> T}\frac{1}{(m_1+n_1)^{2\sigma}}\\
&\ll \sum_{m_1\geq 1\atop {m_2\geq 1}}\frac{1}{(m_1m_2)^{\sigma_0}} \int_{T}^\infty \frac{du}{(m_1+u)^{2\sigma}}\\
&\ll \sum_{m_1\geq 1\atop {m_2\geq 1}}\frac{1}{(m_1m_2)^{\sigma_0}(m_1+T)^{2\sigma-1}}\ll T^{1-2\sigma}.
\end{align*}
Similarly we obtain $\overline{U_1}+U_3, U_3 \ll T^{1-2\sigma}$. 
Therefore we have
\begin{equation}
S_1T=\zeta_2^{[2]}({s_0},2\sigma)T+O\left( T^{2-2\sigma}\right). \label{4-2}
\end{equation}

As for $S_2$, since
$$M(n_1,n_2)=\max\{n_1,n_2,2\}\leq m_1+n_1 (=m_2+n_2),$$
we have
\begin{align*}
S_2& \ll \sum_{m_1\geq 1}\sum_{m_2\geq 1}\frac{1}{(m_1m_2)^{\sigma_0}}\sum_{n_1\leq T\atop {n_2\leq T \atop m_1+n_1=m_2+n_2}}\frac{1}{(m_1+n_1)^{2\sigma-1}}\\
& \ll \sum_{m_1\geq 1}\sum_{m_2\geq 1}\frac{1}{(m_1m_2)^{\sigma_0}}\sum_{n_1\leq T}\frac{1}{(m_1+n_1)^{2\sigma-1}}\\
& \ll \sum_{m_1\geq 1}\frac{1}{m_1^{\sigma_0}}\sum_{m_2\geq 1}\frac{1}{m_2^{\sigma_0}}\sum_{n_1\leq T}\frac{1}{n_1^{2\sigma-1}}\\
& \ll 
\begin{cases}
T^{2-2\sigma} & (1/2<\sigma < 1)\\
\log T & (\sigma=1),
\end{cases}
\end{align*}
because ${\sigma_0}>1$. 

As for $S_3$, we have
\begin{align*}
S_3& \ll \sum_{m_1,m_2\geq 1}\frac{1}{(m_1m_2)^{\sigma_0}}\sum_{n_1, n_2\leq T \atop m_1+n_1<m_2+n_2\leq 2(m_1+n_1)}\frac{1}{(m_1+n_1)^{\sigma}(m_2+n_2)^\sigma}\frac{1}{\log\frac{m_2+n_2}{m_1+n_1}}\\
& \ \ + \sum_{m_1,m_2\geq 1}\frac{1}{(m_1m_2)^{\sigma_0}}\sum_{n_1, n_2\leq T \atop m_2+n_2> 2(m_1+n_1)}\frac{1}{(m_1+n_1)^{\sigma}(m_2+n_2)^\sigma}\frac{1}{\log\frac{m_2+n_2}{m_1+n_1}}.
\end{align*}
We denote the first and the second term by $W_1$ and $W_2$, respectively. 
As for $W_2$, we have
\begin{align*}
W_2& \ll \sum_{m_1,m_2\geq 1}\frac{1}{(m_1m_2)^{\sigma_0}}\sum_{n_1, n_2\leq T \atop m_2+n_2> 2(m_1+n_1)}\frac{1}{(m_1+n_1)^{\sigma}(m_2+n_2)^\sigma}\\
& \ll \sum_{m_1,m_2\geq 1}\frac{1}{(m_1m_2)^{\sigma_0}}\sum_{n_1\leq T}\frac{1}{n_1^{\sigma}}\sum_{n_2\leq T}\frac{1}{n_2^\sigma}\\
& \ll 
\begin{cases}
T^{2-2\sigma} & (1/2<\sigma < 1)\\
(\log T)^2 & (\sigma=1).
\end{cases}
\end{align*}
As for $W_1$, setting $r=(m_2+n_2)-(m_1+n_1)$, we have
\begin{align*}
W_1& \ll \sum_{m_1,m_2\geq 1}\frac{1}{(m_1m_2)^{\sigma_0}}\sum_{n_1\leq T}\sum_{r=1}^{m_1+n_1}\frac{1}{(m_1+n_1)^\sigma(m_1+n_1+r)^\sigma}\frac{1}{\log\left(1+\frac{r}{m_1+n_1}\right)}\\
& \ll \sum_{m_1,m_2\geq 1}\frac{1}{(m_1m_2)^{\sigma_0}}\sum_{n_1\leq T}\frac{1}{(m_1+n_1)^{2\sigma}}\sum_{r=1}^{m_1+n_1}\frac{m_1+n_1}{r}\\
& \ll \sum_{m_1,m_2\geq 1}\frac{1}{(m_1m_2)^{\sigma_0}}\sum_{n_1\leq T}\frac{1}{(m_1+n_1)^{2\sigma-1}}\log(m_1+n_1)\\
& \ll 
\begin{cases}
T^{2-2\sigma}\log T & (1/2<\sigma < 1)\\
(\log T)^2 & (\sigma=1).
\end{cases}
\end{align*}
Combining these results, we obtain
\begin{align*}
& \int_{2}^{T}|\Sigma_1({s_0},s)|^2 dt =\zeta_2^{[2]}({s_0},2\sigma)T+
\begin{cases}
O\left(T^{2-2\sigma}\log T\right) & (1/2<\sigma < 1)\\
O\left((\log T)^2\right) & (\sigma=1).
\end{cases}
\end{align*}
Therefore we have
\begin{align}
& \int_{2}^{T}|\zeta_2({s_0},s)|^2 dt \notag\\
& =\int_{2}^{T}|\Sigma_1({s_0},s)+O\left(t^{-\sigma}\right)|^2 dt \notag\\
& =\int_{2}^{T}|\Sigma_1({s_0},s)|^2 dt+O\left(\int_{2}^{T}|\Sigma_1({s_0},s)|\,t^{-\sigma}dt\right)+O\left(\int_{2}^{T}t^{-2\sigma}dt \right).\label{4-2-2}
\end{align}
We see that the third term on the right-hand side is equal to $O(1)$ because $\frac{1}{2}<\sigma\leq 1$. As for the second term, by the Cauchy-Schwarz inequality, we see that
\begin{align*}
& \int_{2}^{T}|\Sigma_1({s_0},s)|\,t^{-\sigma}dt \\
& \ll \left(\int_{2}^{T}|\Sigma_1({s_0},s)|^2 dt\right)^{1/2}\cdot \left(\int_{2}^{T}\,t^{-2\sigma}dt\right)^{1/2}\\
& =\left( 
\begin{cases}
O(T)+O\left(T^{2-2\sigma}\log T\right) & (1/2<\sigma < 1)\\
O(T)+O\left((\log T)^2\right) & (\sigma=1)
\end{cases}
\right)^{1/2}\cdot O(1)^{1/2}\\
& \ll T^{1/2}.
\end{align*}
This completes the proof of Theorem \ref{T-1-2}.

\end{proof}

\section{The double analogue of the Euler constant} \label{sec-4.25}

Let $\gamma$ be the Euler constant defined by 
$$\gamma=\lim_{N\to \infty}\left( \sum_{n=1}^{N}\frac{1}{n}-\log N\right),$$
which satisfies that
\begin{equation}\label{gamma}
\lim_{s\to 1}\left\{\zeta(s)-\frac{1}{s-1}\right\}=\gamma.
\end{equation}
Here we define analogues of the Euler constant corresponding to the double zeta-function as follows. For ${s_0}\in \mathbb{C}$ with $\Re {s_0}>1$, we let
\begin{equation}
\gamma_2({s_0})=\lim_{N \to \infty}\sum_{m\geq 1}\frac{1}{m^{s_0}}\left\{\sum_{1\leq n\leq N} \frac{1}{(m+n)}-\log (m+N)\right\}. \label{Eu-1}
\end{equation} 
Then we obtain the following.

\begin{prop} \label{P-Euler}
For ${s_0}\in \mathbb{C}$ with $\Re {s_0}>1$,
\begin{equation}
\lim_{s\to 1}\left\{ \zeta_2(s_0,s)-\frac{\zeta(s_0)}{s-1}\right\}=\gamma_2({s_0}).\label{Eu-2}
\end{equation}
In particular,
\begin{equation}
\gamma_2({s_0})=\zeta({s_0})\gamma-\zeta_2(1,{s_0})-\zeta({s_0}+1).\label{Eu-3}
\end{equation}
\end{prop}

\begin{proof}
Applying \eqref{3-4} with $N=0$, we have
\begin{align*}
& \lim_{s\to 1}\left\{ \zeta_2(s_0,s)-\frac{\zeta(s_0)}{s-1}\right\}\\
& =\lim_{s\to 1}\bigg\{ \frac{\sum_{m\geq 1}m^{-{s_0}-s+1}-\zeta({s_0})}{s-1} -s\sum_{m\geq 1}\frac{1}{m^{s_0}}\int_{m}^\infty \frac{u-[u]-1/2}{u^{s+1}}du \\
& \qquad -\frac{1}{2}\sum_{m\geq 1}\frac{1}{m^{{s_0}+s}}\bigg\}\\
&  =\zeta'({s_0})-\sum_{m\geq 1}\frac{1}{m^{{s_0}}}\int_{m}^\infty \frac{u-[u]}{u^{2}}du \\
& \quad +\frac{1}{2}\sum_{m\geq 1}\frac{1}{m^{{s_0}}}\int_{m}^\infty \frac{1}{u^2}du -\frac{1}{2}\zeta({s_0}+1),
\end{align*}
where the third and the fourth terms are cancelled. Hence, from 
$$\zeta'({s_0})=-\sum_{m\geq 1}\frac{\log m}{m^{{s_0}}},$$
the right-hand side of the above equation can be rewritten as 
\begin{align*}
& \zeta'({s_0})-\sum_{m\geq 1}\frac{1}{m^{{s_0}}}\lim_{K\to \infty}\sum_{k=m}^{K+m-1}\int_{k}^{k+1}\frac{u-[u]}{u^{2}}du\\
& =\zeta'({s_0})-\sum_{m\geq 1}\frac{1}{m^{{s_0}}}\lim_{K\to \infty}\sum_{k=m}^{K+m-1}\int_{k}^{k+1}\left(\frac{1}{u}-\frac{k}{u^{2}}\right)du\notag\\
& =\zeta'({s_0})-\sum_{m\geq 1}\frac{1}{m^{{s_0}}}\lim_{K\to \infty}\left(\log(m+K)-\log m-\sum_{k=m}^{m+K-1}\frac{1}{k+1}\right)\notag\\
& =\lim_{K\to \infty}\left\{\sum_{m\geq 1}\frac{1}{m^{{s_0}}}\left(\sum_{n=1}^{K}\frac{1}{m+n}-\log(m+K) \right)\right\}=\gamma_2({s_0}),\notag
\end{align*}
which implies \eqref{Eu-2}.  Note that Arakawa and Kaneko \cite[Proposition 4]{AK}
already showed that $\zeta_2({s_0},s)$, as a function in $s$, has a simple pole at 
$s=1$ with its residue $\zeta({s_0})$, where ${s_0}\in \mathbb{C}$ with $\Re {s_0}>1$. 
Suppose ${s_0}\in \mathbb{C}$ with $\Re {s_0}>1$ and $\Re s>1$. Then it is well-known that
$$\zeta({s_0})\zeta(s)=\zeta_2({s_0},s)+\zeta_2(s,{s_0})+\zeta({s_0}+s).$$
By \eqref{gamma} and \eqref{Eu-2}, we have
\begin{align*}
& \zeta({s_0})\left(\frac{1}{s-1}+\gamma+o(s-1)\right)\\
& \qquad =\left(\frac{\zeta({s_0})}{s-1}+\gamma_2(s_0)+o(s-1)\right)+\zeta_2(s,{s_0})+\zeta({s_0}+s).
\end{align*}
Letting $s\to 1$, we obtain \eqref{Eu-3}.
This completes the proof.
\end{proof}

\section{The second approximation theorem} \label{sec-4.5}

In the previous section, we gave the proof of Theorem \ref{T-1-2} by use of \eqref{4-1} 
which comes from Theorem \ref{T-3-1}. However Theorem \ref{T-3-1} holds under the 
conditions ${\sigma}>0$ and $\sigma_0+\sigma>2$. Hence we cannot use it for 
$3/2<\sigma_0+\sigma\leq 2$.
In order to prove a mean value result in the latter case, we have to prepare another 
approximate formula for $\zeta_2(s_0,s)$.

We begin with \eqref{suzumenooyado}.
As was discussed in the proof of Theorem \ref{T-3-1}, all but the second term on the
right-hand side of \eqref{suzumenooyado} are convergent in 
$\sigma>0, \sigma_0+\sigma>1$, so the remaining task is to study the second term.

%

First we assume ${\sigma_0}+\sigma>2$, $s\neq 1$.  Then by the Euler-Maclaurin
formula we have
\begin{align}
& \frac{1}{1-s}\sum_{m=1}^\infty \frac{1}{m^{s_0} (m+x)^{s-1}} \notag\\
& =\frac{1}{1-s}\int_{1}^\infty \frac{dy}{y^{s_0} x^{s-1}\left(1+\frac{y}{x}\right)^{s-1}} \notag \\
& \ \ +\frac{1}{1-s}\int_{1}^\infty \left(y-[y]-\frac{1}{2}\right)\left(-\frac{{s_0}}{y^{{s_0}+1}(y+x)^{s-1}}+\frac{1-s}{y^{s_0}(y+x)^s}\right)dy \notag\\
& \ \ +\frac{1}{2(1-s)}(1+x)^{1-s}\notag\\
&=g(s_0,s;x)+Y_2+Y_3,\label{5-3}
\end{align}
say.
Obviously $Y_3$ is defined for any $s\in\mathbb{C}\setminus\{1\}$ and satisfies
$Y_3=O\left(t^{-1} x^{1-\sigma}\right)$.
Next consider $Y_2$.   We have
\begin{align*}
&\frac{1}{1-s}\int_{1}^\infty \left(y-[y]-\frac{1}{2}\right)
\frac{{s_0}}{y^{{s_0}+1}(y+x)^{s-1}}dy \ll 
\frac{1}{t}\int_{1}^\infty \frac{dy}{y^{{\sigma_0}+1}(y+x)^{\sigma-1}}\\
&\qquad \ll t^{-1}x^{1-\sigma}\int_{1}^\infty \frac{dy}{y^{{\sigma_0}+1}}
\ll t^{-1}x^{1-\sigma}
\end{align*}
for $\sigma_0>0$, and
\begin{align*}
\frac{1}{1-s} & \int_{1}^\infty \left(y-[y]-\frac{1}{2}\right)\frac{1-s}{y^{s_0}(y+x)^s}dy \ll \int_{1}^\infty \frac{dy}{y^{{\sigma_0}}(y+x)^{\sigma}}\\
& \ll \left(\int_{1}^{x} +\int_{x}^\infty\right)\frac{dy}{y^{{\sigma_0}}(y+x)^{\sigma}}\\
& \ll \int_{1}^{x}\frac{dy}{y^{{\sigma_0}}x^{\sigma}} +\int_{x}^\infty\frac{dy}{y^{{\sigma_0}+\sigma}}\\
& = 
\begin{cases}
O\left(x^{1-{\sigma_0}-\sigma}\right)
    & (0< {\sigma_0}<1;\,{\sigma_0}+\sigma>1)\\
O\left(x^{-\sigma}\log x\right)
    & ({\sigma_0}=1;\,{\sigma_0}+\sigma>1)\\
O\left(x^{-\sigma}\right)
    & ({\sigma_0}>1;\,{\sigma_0}+\sigma>1).
\end{cases}
\end{align*}
Therefore now we find that $Y_2+Y_3$ can be continued to the region $\sigma_0>0$,
${\sigma_0}+\sigma>1$ and $s\neq 1$, and in this region satisfies
\begin{align}\label{Y_2+Y_3}
& Y_2+Y_3 = O(t^{-1}x^{1-\sigma})+
\begin{cases}
O\left(x^{1-{\sigma_0}-\sigma}\right)
   & (0< {\sigma_0}<1;\,{\sigma_0}+\sigma>1)\\
O\left(x^{-\sigma}\log x\right)
   & ({\sigma_0}=1;\,{\sigma_0}+\sigma>1)\\                                              
O\left(x^{-\sigma}\right)                  
    & ({\sigma_0}>1;\,{\sigma_0}+\sigma>1).
\end{cases}
\end{align}

Next we consider $g(s_0,s;x)$.
Here we invoke the classical Mellin-Barnes
integral formula, that is
\begin{align}
  \label{5-4}
   (1+\lambda)^{-s}=\frac{1}{2\pi i}\int_{(c)}\frac{\Gamma(s+z)\Gamma(-z)}
      {\Gamma(s)}\lambda^z dz,
\end{align}
where $s$, $\lambda$ are complex numbers with $\sigma=\Re s>0$, $|\arg\lambda|<\pi$,
$\lambda\neq 0$, $c$ is real with $-\sigma<c<0$, and the path
$(c)$ of integration is the vertical line $\Re z=c$.
(Formula \eqref{5-4} has already been successfully used in the theory of multiple
zeta-functions; see \cite{M1,Mat-JNT,Mat-NMJ}).

\begin{lemma} \label{L-5-2} 
The function $g({s_0},s;x)$ can be continued meromorphically to the region 
$\sigma_0<3/2$ and $\sigma>1/2$, 
and satisfies 
\begin{align*}
g({s_0},s;x)=
\begin{cases}
O\left(t^{-1}x^{1-\sigma}+t^{\sigma_0-2}x^{2-\sigma-\sigma_0}
+t^{-1/2}x^{1/2-\sigma}\right) & (s_0\neq 1)\\
O\left(t^{-1}x^{1-\sigma}(\log t+\log x)+t^{-1/2}x^{1/2-\sigma}\right) & (s_0=1)
\end{cases}
\end{align*}
in this region, except for the points on the singularities
\begin{align}\label{sing}
s=1,\quad s_0+s=2,1,0,-1,-2,-3,-4,\ldots.
\end{align}
\end{lemma}

\begin{proof}
First we assume that ${\sigma_0}>1$ and $\sigma>1$. Then, applying \eqref{5-4} with $\lambda=y/x$ and replacing $s$ by $s-1$ (because $\sigma-1>0$), we have
\begin{equation}
g({s_0},s;x) =\frac{1}{(2\pi i)(1-s)}\int_{1}^\infty \frac{1}{y^{s_0} x^{s-1}}\int_{(c)}\frac{\Gamma(s-1+z)\Gamma(-z)}{\Gamma(s-1)}\left( \frac{y}{x}\right)^z dz\,dy,
\label{5-7}
\end{equation}
where $1-\sigma<c<0$. Here we see that it is possible to change the order of the integral as follows. Since $1-\sigma<c<0<{\sigma_0}-1$, we have $-{\sigma_0}+c<-1$. This implies that \eqref{5-7} is absolutely convergent with respect to $y$. Moreover, by the Stirling formula, we can easily check that \eqref{5-7} is absolutely convergent with respect to $z$. Therefore, changing the order of the integral on the right-hand side of \eqref{5-7}, we obtain
\begin{align*}
g({s_0},s;x) &=\frac{x^{1-s}}{(2\pi i)(1-s)\Gamma(s-1)}\int_{(c)}{\Gamma(s-1+z)\Gamma(-z)}x^{-z}\int_{1}^\infty y^{z-{s_0}} dy\,dz\\
&=\frac{x^{1-s}}{(2\pi i)(1-s)\Gamma(s-1)}\int_{(c)}\frac{\Gamma(s-1+z)\Gamma(-z)}{x^{z}({s_0}-1-z)} dz.
\end{align*}
Now we temporarily assume that $1<\sigma_0<3/2$.   Then the pole $z=s_0-1$
of the integrand is located in the strip $c<\Re z<1/2$.
We shift the path $(c)$ to $\Re z=1/2$. 
Relevant poles are at $z=0$ and $z={s_0}-1$.
Counting the residues of those poles, we obtain
\begin{align}
g({s_0},s;x)& =\frac{x^{1-s}}{(1-s)\Gamma(s-1)}\bigg\{ \frac{\Gamma(s-1)}{{s_0}-1}+\frac{\Gamma(s+{s_0}-2)\Gamma(1-{s_0})}{x^{{s_0}-1}}\notag\\
& \qquad+\frac{1}{(2\pi i)}\int_{(1/2)}\frac{\Gamma(s-1+z)\Gamma(-z)}{x^{z}({s_0}-1-z)} dz\bigg\}\notag\\
&=\frac{x^{1-s}}{(1-s)({s_0}-1)}+\frac{x^{1-s}}{(1-s)\Gamma(s-1)}\frac{\Gamma(s+{s_0}-2)\Gamma(1-{s_0})}{x^{{s_0}-1}}\notag\\
& +\frac{x^{1-s}}{(2\pi i)(1-s)\Gamma(s-1)}\int_{(1/2)}\frac{\Gamma(s-1+z)\Gamma(-z)}{x^{z}({s_0}-1-z)} dz\notag\\
&=R_1+R_2+R_3,\label{R-tachi}
\end{align}
say.   The last integral can be holomorphically continued to the region
$\sigma_0<3/2$ and $\sigma>1/2$ (because in this region the path does not meet the
poles of the integrand).
Therefore \eqref{R-tachi} gives the meromorphic
continuation of $g({s_0},s;x)$ to this region.
The possible singularities of $R_1$ and $R_2$ are $s_0=1$ and those listed as 
\eqref{sing}.   But $s_0=1$ is actually not a singularity.   Putting $s_0=1+\delta$
and calculating the limit $\delta\to 0$, we find that
\begin{align}\label{at_1}
\biggl.R_1+R_2\biggr|_{s_0=1}&=\frac{x^{1-s}}{1-s}\left(\log x-\gamma-
\frac{\Gamma'}{\Gamma}(s-1)\right).
\end{align}
We can easily check that $R_1=O\left(t^{-1}x^{1-\sigma}\right)$ and 
$R_2=O\left(t^{\sigma_0-2}x^{2-\sigma-\sigma_0}\right)$ 
by the Stirling formula, if $s_0\neq 1$ and 
$(s_0,s)$ is not on the singularities \eqref{sing}.
If $s_0=1$, then from \eqref{at_1} we see that 
$$
R_1+R_2=O\left(t^{-1}x^{1-\sigma}(\log t+\log x)\right).
$$
As for $R_3$, setting $z=1/2+iy$, we have
\begin{align*}
R_3&\ll \frac{x^{1-\sigma}e^{\pi t/2}}{t\cdot t^{\sigma-3/2}}\int_{-\infty}^\infty 
\left|\frac{\Gamma(\sigma-1+it+1/2+iy)\Gamma(-1/2-iy)}{x^{1/2+iy}({\sigma_0}+it_0-1-1/2-iy)}\right|dy \\
&\ll (tx)^{1/2-\sigma}e^{\pi t/2}\int_{-\infty}^\infty (|t+y|+1)^{\sigma-1}
e^{-\pi|t+y|/2}(|y|+1)^{-2}e^{-\pi |y|/2}dy.
\end{align*}
By Lemma 4 of \cite{Mat-JNT}, we find that the above integral is
$O(t^{\sigma-1}e^{-\pi t/2})$, and hence $R_3=O(t^{-1/2}x^{1/2-\sigma})$.
This completes the proof of Lemma \ref{L-5-2}.
\end{proof}

\begin{remark}
By shifting the path more to the right, it is possible to prove that $g({s_0},s;x)$
can be continued meromorphically to the whole space $\mathbb{C}^2$.
\end{remark}

From \eqref{Y_2+Y_3} and Lemma \ref{L-5-2} we find that the right-hand side of
\eqref{5-3} can be continued to the region $\sigma_0<3/2$, $\sigma>1/2$, 
$\sigma_0+\sigma>1$, and satisfies the estimates proved above.
On the other hand, the last three terms on the right-hand side of
\eqref{suzumenooyado} are estimated by \eqref{S_4}, \eqref{S_3}, and 
\eqref{E}, respectively.   

Now set $x=t$.   Then, using \eqref{E} we have
\begin{align*}
\sum_{m=1}^{\infty}\frac{E(s;x,m,N)}{m^{s_0}}
\ll\sum_{m=1}^{\infty}\frac{1}{m^{\sigma_0}(m+t)^{\sigma}}
&\ll\sum_{m\leq t} \frac{1}{m^{\sigma_0}t^{\sigma}}
              +\sum_{m> t}\frac{1}{m^{\sigma_0+\sigma}}\\
&\ll\left\{
\begin{array}{ll}
 t^{1-\sigma_0-\sigma} & (0<\sigma_0<1)\\
 t^{-\sigma}\log t     & (\sigma_0=1)\\
 t^{-\sigma}           & (\sigma_0>1),
\end{array}\right.
\end{align*}
while \eqref{S_4} and \eqref{S_3} imply that the contributions of $A_3$ and $A_4$
vanish when $N\to\infty$.

Collecting all the information, we obtain the following.

\begin{theorem} \label{T-5-3} 
Let ${s_0}={\sigma_0}+i{t_0}\in \mathbb{C}$ with $0< {\sigma_0}< 3/2$ and 
$s=\sigma+it \in \mathbb{C}$ with $\sigma>1/2$, ${\sigma_0}+\sigma>1$,
$s\neq 1$, and $s_0+s\neq 2$.  Then 
\begin{equation}
\begin{split}
& \zeta_2({s_0},s)= \sum_{m=1}^\infty \sum_{n\leq t}\frac{1}{m^{s_0} (m+n)^s}+
\begin{cases}
O\left(t^{1-{\sigma_0}-\sigma}\right)& ({\sigma_0}<1 )\\
O\left(t^{-\sigma}\log t\right)& ({\sigma_0}=1)\\
O\left(t^{-\sigma}\right) & ({\sigma_0}>1).
\end{cases}
\end{split}
\label{5-8}
\end{equation}
\end{theorem}

\section{Proof of Theorem \ref{T-1-3}} \label{sec-5}

Based on these results, we finally give the proof of Theorem \ref{T-1-3}.

\begin{proof}[Proof of Theorem \ref{T-1-3}] 
We let ${s_0}\in \mathbb{C}$ with $1/2<{\sigma_0}<3/2$ and 
$s\in \mathbb{C}$ with $1/2<\sigma\leq 1$ and $3/2<{\sigma_0}+\sigma\leq 2$. 
We further assume that $s_0+s\neq 2$. 
Similarly to Section \ref{sec-4}, let 
$$\Sigma_1({s_0},s)=\sum_{m=1}^\infty \sum_{1\leq n \leq t}\frac{1}{m^{{s_0}}(m+n)^{s}}.$$
Then we can again obtain \eqref{mean-val} and denote it by $S_1T-S_2+S_3$.
As for $S_1$, we similarly set $k=m_1+n_1(=m_2+n_2)$. Then we can write 
\begin{align*}
S_1&= \sum_{k=2}^\infty \left(\sum_{m_1=1}^{k-1} \sum_{m_2=1}^{k-1}\frac{1}{m_1^{s_0}m_2^{\overline{s_0}}}\right)\frac{1}{k^{2\sigma}} -(U_1+U_2+U_3),
\end{align*}
where $\overline{U_1}=U_2$.  We have
\begin{align*}
U_1+U_3& \ll \sum_{m_1\geq 1\atop {m_2\geq 1}}\frac{1}{(m_1m_2)^{\sigma_0}} \sum_{n_1> T}\frac{1}{(m_1+n_1)^{2\sigma}}\\
&\ll \sum_{k=2}^\infty \frac{1}{k^{2\sigma}}\sum_{m_1=1 \atop k-m_1>T}^{k-1}\sum_{m_2=1}^{k-1}\frac{1}{(m_1m_2)^{\sigma_0}},
\end{align*}
where we set $k=m_1+n_1=m_2+n_2$. Note that from the condition $k-m_1>T$, we have $k>T$. Hence we obtain
\begin{align*}
U_1+U_3&\ll \sum_{k>T}\frac{1}{k^{2\sigma}}\int_{1}^{k}u^{-{\sigma_0}}du \int_{1}^{k}v^{-{\sigma_0}}dv\\
&\ll 
\begin{cases}
\sum_{k>T}k^{2-2\sigma-2{\sigma_0}}=O\left(T^{3-2\sigma-2{\sigma_0}}\right) & (\frac{1}{2}<{\sigma_0}<1)\\
\sum_{k>T}k^{-2\sigma}(\log k)^2=O\left(T^{1-2\sigma}(\log T)^2\right) & ({\sigma_0}=1)\\
\sum_{k>T}k^{-2\sigma}=O\left(T^{1-2\sigma}\right) & (1<\sigma_0<\frac{3}{2}),
\end{cases}
\end{align*}
because $2-2\sigma-2{\sigma_0}<-1$, and the same estimate holds for $\overline{U_1}+U_3$ and $U_3$. 
As for the second estimate we used the integration by parts for 
$$\sum_{k>T}k^{-2\sigma}(\log k)^2\ll\int_{T}^\infty u^{-2\sigma}(\log u)^2 du.$$
Therefore
\begin{equation}
S_1=\zeta_2^{[2]}({s_0},2\sigma)+
\begin{cases}
O\left(T^{3-2\sigma-2{\sigma_0}}\right) & (\frac{1}{2}<{\sigma_0}<1)\\
O\left(T^{1-2\sigma}(\log T)^2\right) & ({\sigma_0}=1)\\
O\left(T^{1-2\sigma}\right) & (1<\sigma_0<\frac{3}{2}).
\end{cases}
\label{5-9}
\end{equation}

Next we consider $S_2$. Using $M(n_1,n_2)=\max\{n_1,n_2,2\}$, we have 
\begin{align*}
S_2&= \sum_{m_1\geq 1 \atop m_2\geq 1}\frac{1}{m_1^{s_0}m_2^{\overline{s_0}}}\sum_{n_1\leq T \atop {n_2\leq T \atop m_1+n_1=m_2+n_2}}\frac{M(n_1,n_2)}{(m_1+n_1)^{2\sigma}}\\
&= \sum_{k=2}^\infty \frac{1}{k^{2\sigma}}\sum_{m_1\geq 1 \atop m_2\geq 1}\sum_{n_1\leq T \atop {n_2\leq T \atop {m_1+n_1=k \atop m_2+n_2=k}}}\frac{M(n_1,n_2)}{m_1^{s_0}m_2^{\overline{s_0}}}\\
& \ll \sum_{k\leq T} \frac{k}{k^{2\sigma}}\sum_{m_1=1}^{k-1} \sum_{m_2=1}^{k-1}\frac{1}{(m_1m_2)^{\sigma_0}}+\sum_{k> T} \frac{T}{k^{2\sigma}}\sum_{m_1=1}^{k-1} \sum_{m_2=1}^{k-1}\frac{1}{(m_1m_2)^{\sigma_0}}\\
& \ll 
\begin{cases}
\sum_{k\leq T} {k^{1-2\sigma}}\left(k^{1-{\sigma_0}}\right)^2+T\sum_{k> T} {k^{-2\sigma}}\left(k^{1-{\sigma_0}}\right)^2 &  (\frac{1}{2}<{\sigma_0}<1)\\
\sum_{k\leq T} {k^{1-2\sigma}}\left(\log k\right)^2+T\sum_{k> T} {k^{-2\sigma}}\left(\log k\right)^2 &  ({\sigma_0}=1)\\
\sum_{k\leq T} {k^{1-2\sigma}}+T\sum_{k> T} {k^{-2\sigma}} & (1<\sigma_0<\frac{3}{2}).
\end{cases}
\end{align*}
Therefore we obtain 
\begin{equation}
S_2=
\begin{cases}
O\left(T^{4-2{\sigma_0}-2\sigma}\right)&  
    (\frac{1}{2}<{\sigma_0}<1,\frac{1}{2}<\sigma\leq 1)\\
O\left(T^{2-2\sigma}(\log T)^2\right)&  ({\sigma_0}=1,\frac{1}{2}<\sigma<1)\\
O\left((\log T)^3\right)& ({\sigma_0}=1,\sigma=1)\\
O\left(T^{2-2\sigma}\right)& (1<\sigma_0<\frac{3}{2},\frac{1}{2}<\sigma<1),
\end{cases}
\label{5-10}
\end{equation}
where we have to note that $3/2<\sigma_0+\sigma<2$ in the first case, and
$\sigma\neq 1$ (because if $\sigma=1$ then $\sigma_0+\sigma>2$) in the fourth case.

Finally we consider $S_3$.
Similarly to the argument in Section \ref{sec-4}, we have
\begin{align*}
S_3& \ll \sum_{m_1,m_2\geq 1}\frac{1}{(m_1m_2)^{\sigma_0}}\sum_{n_1, n_2\leq T \atop m_1+n_1<m_2+n_2\leq 2(m_1+n_1)}\frac{1}{(m_1+n_1)^{\sigma}(m_2+n_2)^\sigma}\frac{1}{\log\frac{m_2+n_2}{m_1+n_1}}\\
& \ \ + \sum_{m_1,m_2\geq 1}\frac{1}{(m_1m_2)^{\sigma_0}}\sum_{n_1, n_2\leq T \atop m_2+n_2> 2(m_1+n_1)}\frac{1}{(m_1+n_1)^{\sigma}(m_2+n_2)^\sigma}\frac{1}{\log\frac{m_2+n_2}{m_1+n_1}},
\end{align*}
which we denote by $W_1+W_2$. 

First estimate $W_2$.    We have
\begin{align*}
W_2& \ll \sum_{m_1,m_2\geq 1}\frac{1}{(m_1m_2)^{\sigma_0}}\sum_{n_1, n_2\leq T \atop m_2+n_2> 2(m_1+n_1)}\frac{1}{(m_1+n_1)^{\sigma}(m_2+n_2)^\sigma}\\
&=\sum_{m_1\geq 1 \atop n_1\leq T}\frac{1}{m_1^{\sigma_0}(m_1+n_1)^{\sigma}}           
   \sum_{m_2\geq 1, n_2\leq T \atop m_2+n_2> 2(m_1+n_1)}\frac{1}                        
   {m_2^{\sigma_0}(m_2+n_2)^{\sigma}}\\
&=\sum_{m_1\geq 1 \atop n_1\leq T}\frac{1}{m_1^{\sigma_0}(m_1+n_1)^{\sigma}}
   \sum_{k>2(m_1+n_1)}\frac{1}{k^{\sigma}}\sum_{m_2\geq 1, n_2\leq T \atop
   m_2+n_2=k}\frac{1}{m_2^{\sigma_0}}\\
&= \sum_{m_1\leq T \atop n_1\leq T}+\sum_{m_1> T \atop n_1\leq T}=W_{21}+W_{22},
\end{align*}
say.
Consider $W_{22}$.    Since $m_1> T$, we have $k>2T$, so 
$m_2=k-n_2\geq k-T>k/2$.   Therefore the innermost sum of $W_{22}$ is
$$
\sum_{k-T\leq m_2\leq k-1}\frac{1}{m_2^{\sigma_0}}\ll T k^{-\sigma_0},
$$
and hence
\begin{align*}
W_{22}&\ll T\sum_{m_1> T \atop n_1\leq T}\frac{1}{m_1^{\sigma_0}(m_1+n_1)^{\sigma}}   
   \sum_{k>2(m_1+n_1)}k^{-\sigma_0-\sigma}\\
&\ll T\sum_{m_1> T}m_1^{-\sigma_0}\sum_{n_1\leq T}(m_1+n_1)^{1-\sigma_0-2\sigma}\\
&\leq T\sum_{m_1> T}m_1^{-\sigma_0}\sum_{n_1\leq T}m_1^{1-\sigma_0-2\sigma}\\
&\ll T^2 \sum_{m_1> T}m_1^{1-2\sigma_0-2\sigma}.
\end{align*}
Since $1-2\sigma_0-2\sigma<-1$, we have
\begin{align}\label{W_{22}}
W_{22}\ll T^2 T^{2-2\sigma_0-2\sigma}&=T^{4-2\sigma_0-2\sigma}.
\end{align}
As for $W_{21}$, we further divide the inner double sum of $W_{21}$ into two parts
$D_1$ and $D_2$ according to $2(m_1+n_1)<k\leq 2T$ and $k>2T$, respectively.
We handle the innermost sum of $D_2$ similarly to the case of $W_{22}$.   We have
\begin{align*}
D_2\ll T\sum_{k>2T}\frac{1}{k^{\sigma}}\frac{1}{k^{\sigma_0}}
   \ll T^{2-\sigma_0-\sigma}.
\end{align*}
The innermost sum of $D_1$ is
\begin{align*}
\ll\sum_{m_2\leq k-1}\frac{1}{m_2^{\sigma_0}}
\ll
\begin{cases}
k^{1-\sigma_0} & (\frac{1}{2}<\sigma_0<1)\\
\log k & (\sigma_0=1)\\
1 & (1<\sigma_0<\frac{3}{2}),
\end{cases}
\end{align*}
which gives
\begin{align*}
D_1\ll
\begin{cases} 
T^{2-\sigma_0-\sigma} & (\frac{1}{2}<\sigma_0<1, \frac{1}{2}<\sigma\leq 1)\\
T^{1-\sigma}\log T & (\sigma_0=1,\frac{1}{2}<\sigma<1)\\
(\log T)^2 & (\sigma_0=1,\sigma=1)\\
T^{1-\sigma} & (1<\sigma_0<\frac{3}{2}, \frac{1}{2}<\sigma<1).
\end{cases}
\end{align*}
Substituting the estimates of $D_1$ and $D_2$ into $W_{21}$, and estimating the
remaining sum
$$
\sum_{m_1\leq T \atop n_1\leq T}\frac{1}{m_1^{\sigma_0}(m_1+n_1)^{\sigma}}
\leq \sum_{m_1\leq T}\frac{1}{m_1^{\sigma_0}}\sum_{n_1\leq T}\frac{1}{n_1^{\sigma}}
$$
in the obvious way, we obtain
\begin{align}\label{W_{21}}
W_{21}&\ll
\begin{cases}
T^{4-2\sigma_0-2\sigma} & (\frac{1}{2}<\sigma_0<1,\frac{1}{2}<\sigma<1)\\
T^{2-2\sigma_0}\log T   & (\frac{1}{2}<\sigma_0<1,\sigma=1)\\
T^{2-2\sigma}(\log T)^2 & (\sigma_0=1,\frac{1}{2}<\sigma<1)\\
(\log T)^4              & (\sigma_0=1,\sigma=1)\\
T^{2-2\sigma}           & (1<\sigma_0<\frac{3}{2},\frac{1}{2}<\sigma<1).
\end{cases}
\end{align}

Next consider $W_1$. We have
\begin{align*}
W_1&= \sum_{m_1,m_2\geq 1}\frac{1}{(m_1m_2)^{\sigma_0}}\sum_{n_1,\,n_2\leq T\atop {m_1+n_1<m_2+n_2 \atop \leq 2(m_1+n_1)}}\frac{1}{(m_1+n_1)^\sigma(m_2+n_2)^\sigma}\frac{1}{\log\left(1+\frac{m_2+n_2-m_1-n_1}{m_1+n_1}\right)}\\
& \ll \sum_{m_1\geq 1}\sum_{n_1\leq T}\frac{1}{m_1^{\sigma_0}(m_1+n_1)^{2\sigma}}\sum_{n_2\leq T}\sum_{m_2\geq 1 \atop {m_1+n_1<m_2+n_2 \atop \leq 2(m_1+n_1)}}\frac{m_1+n_1}{m_2^{\sigma_0}(m_2+n_2-m_1-n_1)}\\
& =\sum_{m_1\leq 2T}\sum_{n_1\leq T}+\sum_{m_1> 2T}\sum_{n_1\leq T}=W_{11}+W_{12},
\end{align*}
say. Consider $W_{12}$. Since $m_1>2T$, we have $n_2\leq T<m_1/2$, so $m_2>m_1+n_1-n_2>m_1/2$. Therefore, setting $r=m_2+n_2-m_1-n_1$, we have
\begin{align}
W_{12} & \ll \sum_{m_1\geq 2T}\sum_{n_1\leq T}\frac{1}{m_1^{2\sigma_0}(m_1+n_1)^{2\sigma-1}}\sum_{n_2\leq T}\sum_{m_2\geq 1 \atop {m_1+n_1<m_2+n_2 \atop \leq 2(m_1+n_1)}}\frac{1}{(m_2+n_2-m_1-n_1)}\notag\\
& \ll \sum_{m_1\geq 2T}\sum_{n_1\leq T}\frac{1}{m_1^{2\sigma_0}(m_1+n_1)^{2\sigma-1}}\sum_{n_2\leq T}\sum_{r=1}^{m_1+n_1}\frac{1}{r}\notag\\
& \ll T\sum_{m_1\geq 2T}\sum_{n_1\leq T}\frac{1}{m_1^{2\sigma_0}(m_1+n_1)^{2\sigma-1}}\log(m_1+n_1)\notag\\
&\ll T\sum_{m_1>2T}m_1^{-2\sigma_0}\times                                               
\begin{cases}                                                                           
(m_1+T)^{2-2\sigma}\log(m_1+T) & (\frac{1}{2}<\sigma<1)\\                               
(\log(m_1+T))^2 & (\sigma=1)                                                            
\end{cases}                                                                             
\notag\\                                                                                
&\ll                                                                                    
\begin{cases}                                                                           
T^{4-2\sigma_0-2\sigma}\log T & (\frac{1}{2}<\sigma<1)\\                                
T^{2-2\sigma_0}(\log T)^2 & (\sigma=1).                                                 
\end{cases}                                                                     \label{W_{12}}        
\end{align}
Next, 
since $m_2<2(m_1+n_1)$, the innermost sum of $W_{11}$ is
\begin{align*}
&\ll \sum_{m_2<2(m_1+n_1)}m_2^{-\sigma_0}\log(m_1+n_1)\\
&\ll
\begin{cases}
(m_1+n_1)^{1-\sigma_0}\log(m_1+n_1) & (\frac{1}{2}<\sigma_0<1)\\
(\log(m_1+n_1))^2 & (\sigma_0=1)\\
\log(m_1+n_1) & (1<\sigma_0<\frac{3}{2}).
\end{cases}
\end{align*}
Therefore, when $\frac{1}{2}<\sigma_0<1$, we have
\begin{align*}
W_{11}&\ll \sum_{m_1\leq 2T \atop n_1\leq T}m_1^{-\sigma_0}
     (m_1+n_1)^{2-\sigma_0-2\sigma}\log(m_1+n_1)\\
&\ll \sum_{m_1\leq 2T} m_1^{-\sigma_0}(m_1+T)^{3-\sigma_0-2\sigma}\log(m_1+T)\\
&\ll T^{4-2\sigma_0-2\sigma}\log T,
\end{align*}
because $2<\sigma_0+2\sigma<3$.   Similarly, we have
\begin{align*}
W_{11}&\ll \sum_{m_1\leq 2T \atop n_1\leq T}m_1^{-1}(m_1+n_1)^{1-2\sigma}
   (\log(m_1+n_1))^2\\
&\ll
\begin{cases}
T^{2-2\sigma}(\log T)^3 & (\frac{1}{2}<\sigma<1)\\
(\log T)^4 & (\sigma=1)
\end{cases}
\end{align*}
when $\sigma_0=1$, and
\begin{align*}
W_{11}\ll\sum_{m_1\leq 2T \atop n_1\leq T}m_1^{-\sigma_0}(m_1+n_1)^{1-2\sigma}
    \log(m_1+n_1)
\ll T^{2-2\sigma}\log T
\end{align*}
when $1<\sigma_0<\frac{3}{2}$.
By \eqref{W_{22}}, \eqref{W_{21}}, \eqref{W_{12}} and the above estimates, we
now obtain
\begin{align}
S_3&=W_1+W_2\notag\\
&\ll
\begin{cases}
T^{4-2\sigma_0-2\sigma}\log T & (\frac{1}{2}<\sigma_0<1,\frac{1}{2}<\sigma<1)\\        
T^{2-2\sigma_0}(\log T)^2   & (\frac{1}{2}<\sigma_0<1,\sigma=1)\\                       
T^{2-2\sigma}(\log T)^3 & (\sigma_0=1,\frac{1}{2}<\sigma<1)\\                           
(\log T)^4              & (\sigma_0=1,\sigma=1)\\                                       
T^{2-2\sigma}\log T           & (1<\sigma_0<\frac{3}{2},\frac{1}{2}<\sigma<1).          
\end{cases}                                                                     \label{SSS}
\end{align}
Denote the right-hand side of the above by $\mathcal{E}(T)$.
Combining \eqref{5-9}, \eqref{5-10} and \eqref{SSS}, we obtain
\begin{align}\label{5-12}
\int_2^{T}|\Sigma_1({s_0},s)|^2dt=S_1T-S_2+S_3 &
=\zeta_2^{[2]}({s_0},2\sigma)T+O(\mathcal{E}(T)).
\end{align}
Now, using the Cauchy-Schwarz inequality, we estimate
the second term on the right-hand side of \eqref{4-2-2} with replacing $t^{-\sigma}$ by the error term 
on the right-hand side of \eqref{5-8}. 
Denoting by $E(t)$ the error term 
on the right-hand side of \eqref{5-8}, 
we have
\begin{align*}
& \int_{2}^{T}|\Sigma_1({s_0},s)|\,E(t)dt  \ll \left(\int_{2}^{T}|\Sigma_1({s_0},s)|^2 dt\right)^{1/2}\cdot \left(\int_{2}^{T}\,E(t)^2 dt\right)^{1/2}\\
& \ =
\left\{ O(T)+O(\mathcal{E}(T))\right\}^{1/2}\cdot O(1)^{1/2}\ll T^{1/2}.
\end{align*}
Thus we obtain the proof of Theorem \ref{T-1-3}.
\end{proof}

\

{\sc Acknowledgements.} \ 
 The authors express their gratitude to 
 Mr. Soichi Ikeda, Mr. Kaneaki Matsuoka, 
 Mr. Akihiko Nawashiro and Mr. Tomokazu Onozuka for pointing out 
 some inaccuracies included in the original version of the manuscript.

\

\bibliographystyle{amsplain}

\begin{thebibliography}{999}

\bibitem{AET}
S.~Akiyama, S.~Egami and Y.~Tanigawa, {Analytic continuation of multiple 
zeta-functions and their values at non-positive integers}, Acta Arith. 
\textbf{98} (2001), 107-116.

\bibitem{AK}
T.~Arakawa and M.~Kaneko, {Multiple zeta values, poly-Bernoulli numbers,
  and related zeta functions}, Nagoya Math. J. \textbf{153} (1999), 189--209.

\bibitem{Atk}
F. V. Atkinson, {The mean-value of the Riemann zeta function}, 
Acta Math. \textbf{81} (1949), 353-376.

\bibitem{BBG}
D.~Borwein, J.~M. Borwein, and R.~Girgensohn, {Explicit evaluation of
  Euler sums}, Proc. Edinburgh Math. Soc. \textbf{38} (1995), 277--294.

\bibitem{IM}
H. Ishikawa and K. Matsumoto, {On the estimation of the order of Euler-Zagier
multiple zeta-functions}, Illinois J. Math. \textbf{47} (2003), 1151-1166.

\bibitem{KT}
I. Kiuchi and Y. Tanigawa, {Bounds for double zeta-functions},
Ann. Sc. Norm. Sup. Pisa, Cl. Sci. Ser. V \textbf{5} (2006), 445-464.

\bibitem{KTZ}
I. Kiuchi, Y. Tanigawa and W. Zhai, {Analytic properties of double
zeta-functions}, Indag. Math. \textbf{21} (2011), 16-29.

\bibitem{KMT-Debrecen}
Y. Komori, K. Matsumoto and H. Tsumura,
{Functional equations and functional relations for the Euler double
zeta-function and its generalization of Eisenstein type}, Publ. Math. Debrecen
\textbf{77} (2010), 15-31.

\bibitem{M1}
K.~Matsumoto, {On the analytic continuation of various multiple
  zeta-functions}, In: Number Theory for the Millennium, II (Urbana, IL, 2000), A K Peters, Natick, MA, 2002, pp.~417--440.

\bibitem{Mat-JNT}
\bysame, {The analytic continuation and the asymptotic behaviour of 
  certain multiple zeta-functions I}, J. Number Theory \textbf{101} (2003),
  223--243.

\bibitem{Mat-NMJ}
\bysame, {Asymptotic expansions of double zeta-functions of Barnes, of 
  Shintani, and Eisenstein series}, Nagoya Math. J. \textbf{172} (2003),
  59--102.

\bibitem{Mat-Camb}
\bysame, {Functional equations for double zeta-functions},
Math. Proc. Cambridge Phil. Soc. \textbf{136} (2004), 1-7.

\bibitem{MT-RIMS}
K.~Matsumoto and H.~Tsumura, Mean value theorems for double zeta-functions, 
In: Analytic Number Theory -- Number Theory through Approximation and Asymptotics, RIMS, 2012, (ed. K. Chinen), RIMS 
K$\hat{\rm o}$ky$\hat{\rm u}$roku {\bf 1874}, RIMS, 2014, pp.~45-54.

\bibitem{Nak-Pan}
T. Nakamura and {\L}. Pa{\'n}kowski, {Any non-monomial polynomial of the Riemann zeta-function has complex zeros off the critical line}, preprint, arXiv:1212.5890.

\bibitem{Titch}
E.~C. Titchmarsh, {The {T}heory of the {R}iemann {Z}eta-function, {S}econd
  {E}dition, {E}dited and with a preface by {D}. {R}. {H}eath-{B}rown}, The
  Clarendon Press, Oxford University Press, New York, 1986.

\end{thebibliography}

\end{document}